\documentclass[11pt]{article}
\usepackage{fullpage,amsthm,amssymb,tikz}
\usepackage[font=footnotesize]{caption} 
\usepackage{moresize}
\def\vph{\varphi}
\def\C{\mathcal{C}}
\def\Abad{A_{\mathrm{\footnotesize{bad}}}}
\def\dist{{\textrm{dist}}}
\def\Xlow{X_{\mathrm{low}}}
\def\Xhigh{X_{\mathrm{high}}}
\def\Ylow{Y_{\mathrm{low}}}
\def\Yhigh{Y_{\mathrm{high}}}
\newtheorem*{mainthm}{Main Theorem}
\newtheorem{lem}{Lemma}
\newtheorem{conj}{Conjecture}

\RequirePackage{marginnote,hyperref}
\newcommand{\aside}[1]{\marginnote{\scriptsize{#1}}[0cm]}
\newcommand{\aaside}[2]{\marginnote{\scriptsize{#1}}[#2]}
\newcommand\Emph[1]{\emph{#1}\aside{#1}}

\author{Daniel W. Cranston\thanks{%
Department of Computer Science, Virginia Commonwealth
University, Richmond, VA, USA;
\texttt{dcranston@vcu.edu}
}}
\begin{document}
\title{Strong Edge-Coloring of Cubic\\ Bipartite Graphs: A Counterexample}
\maketitle
\abstract{A strong edge-coloring $\vph$ of a graph $G$ assigns colors to edges of
$G$ such that $\vph(e_1)\ne \vph(e_2)$ whenever $e_1$ and $e_2$ are at distance
no more than 1.  It is equivalent to a proper vertex coloring of the square of the line
graph of $G$.  In 1990 Faudree, Schelp, Gy\'arf\'as, and Tuza conjectured that if
$G$ is a bipartite graph with maximum degree 3 and sufficiently large girth,
then $G$ has a strong edge-coloring with at most 5 colors.  
In 2021 this conjecture was disproved by Lu\v{z}ar, Ma\v{c}ajov\'{a},
\v{S}koviera, and Sot\'{a}k.  Here we give an alternative construction to
disprove the conjecture.
}
\bigskip

\section{Introduction}
A \Emph{strong edge-coloring} $\vph$ of a graph $G$ assigns colors to the edges
of $G$ such that $\vph(e_1)\ne\vph(e_2)$ whenever $e_1$ and $e_2$ are at distance
no more than 1.  (This is equivalent to a proper vertex coloring of the square
of the line graph.)  The \emph{strong chromatic index} of $G$, denoted
\Emph{$\chi'_s(G)$} is the smallest number of colors that admits a strong
edge-coloring.  This notion was introduced in 1983 by Fouquet and
Jolivet~\cite{FJ2,FJ1}.  In 1985 Erd\H{o}s and Ne\v{s}et\v{r}il conjectured,
for every graph $G$ with maximum degree $\Delta$,
that $\chi'_s(G)\le \frac54\Delta^2$ and that the lower order terms can be
improved slightly when $\Delta$ is odd.  This problem has spurred much work in
the area, and
Deng, Yu, and Zhou~\cite{DYZ-survey} survey results
through 2019.  In this note we focus on a conjecture from 1990 of
Faudree, Schelp, Gy\'arf\'as, and Tuza~\cite{FSGT}.
\begin{conj}[\cite{FSGT}]
\label{conj1}
Let $G$ be a graph with $\Delta(G)=3$.
\begin{enumerate}
\item[(1)] Now $\chi'_s(G)\le 10$.
\item[(2)] If $G$ is bipartite, then $\chi'_s(G)\le 9$.
\item[(3)] If $G$ is planar, then $\chi'_s(G)\le 9$.
\item[(4)] If $G$ is bipartite and for each edge $xy\in E(G)$ we have $d(x)+d(y)\le
5$, then $\chi'_s(G)\le 6$.
\item[(5)] If $G$ is bipartite and has no 4-cycle, then $\chi'_s(G)\le 7$.
\item[(6)] If $G$ is bipartite and its girth is large, then $\chi'_s(G)\le 5$.
\end{enumerate}
\end{conj}
Four parts of this conjecture have been confirmed.  In the early 1990s
Andersen~\cite{Andersen} and Hor\'{a}k,
Qing, and Trotter~\cite{HQT} proved (1).  In 1993 Steger and Yu~\cite{SY} proved (2).
In 2016 Kostochka, Li, Ruksasakchai, Santana, Wang, and Yu~\cite{KLRSWY} proved (3). And
in 2008 Wu and Lin~\cite{WL} proved (4).  As far as we know, (5) remains open.  
In 2021 (6) was disproved by Lu\v{z}ar, Ma\v{c}ajov\'{a}, \v{S}koviera, and
Sot\'{a}k~\cite{LMSS}. 
Here we give an alternate (and, arguably, simpler) construction to disprove (6).

\section{Main Result}
Our Main Theorem is motivated by the special case of $k$-regular graphs where $k=3$, which is all that is
needed to disprove Conjecture~\ref{conj1}(6).  However, with only a bit more work we prove the
result for all $k\ge 2$. 
\begin{mainthm}
For every positive integer $g$ and every integer $k\ge 2$, there exists a $k$-regular bipartite graph $G$
such that $G$ has girth at least $g$ and $\chi'_s(G)\ge 2k$.
\end{mainthm}

We first prove the Main Theorem assuming two lemmas.  We prove the lemmas below.
\begin{proof}
Fix positive integers $g$ and $k\ge 2$.  By Lemma~\ref{lem2}, if $n$ is sufficiently large
then there exists a bipartite $k$-regular graph on $2n$ vertices with
girth at least $g$.  We choose such $n$ that is not divisible by $2k-1$.
Since $G$ is $k$-regular, $|E(G)|=\frac{k}2|V(G)| = kn$.  Since 
$(2k-1)\nmid n$, and $k$ is relatively prime to $2k-1$, also $(2k-1)\nmid\vert
E(G)\vert$.  Thus, Lemma~\ref{lem1} implies that $\chi'_s(G)\ge 2k$.
\end{proof}

We consider an arbitrary edge $e$ in a $k$-regular graph and the $2k-2$
edges that share one endpoint with $e$; in the square of the line graph, the 
corresponding vertices form a clique. So each color in a strong edge-coloring of
$G$ is used on at most one of these $2k-1$ edges.  By repeating
this argument for every edge $e$, and averaging, we deduce that every color in 
a strong edge-coloring is used on at most $1/(2k-1)$ of all edges.  We
formalize this idea below.

\begin{lem}
\label{lem1}
If $G$ is $k$-regular and simple, for some $k\ge 2$, then in every strong
edge-coloring $\vph$ of $G$ every color
class of $\vph$ has size at most $|E(G)|/(2k-1)$.  In particular, if
$(2k-1)\nmid\left|E(G)\right|$, then $\chi'_s(G)\ge 2k$.
\end{lem}
\begin{proof}
Fix a simple $k$-regular graph $G$ and a strong edge-coloring $\vph$ of $G$.  Let $\C$ be a
set of edges receiving the same color under $\vph$.  For each $e\in E(G)$, let
\Emph{$N(e)$} denote the set of edges sharing at least one endpoint with $e$.  Note
that $e\in N(e)$ and $|N(e)|=2k-1$ for every $e\in E(G)$, since $G$ is
$k$-regular.
Furthermore, $e\in N(e')$ for exactly $2k-1$ edges $e'$ (one of which is $e$), for
each $e\in E(G)$.  Since $\vph$ is a strong edge-coloring, we get
$|N(e)\cap \C|\le 1$ for every $e\in E(G)$.  Thus,
$$(2k-1)|\C| = \sum_{e\in E(G)}|\C \cap N(e)|\le \sum_{e\in E(G)}1 = |E(G)|.$$  
So $(2k-1)|\C|\le |E(G)|$, giving $|\C|\le |E(G)|/(2k-1)$.
If also we have $(2k-1)\nmid\left|E(G)\right|$, then $|\C| < |E(G)|/(2k-1)$.  Since $\C$ is 
arbitrary, we get $\chi'_s(G)>|E(G)|/(|E(G)|/(2k-1))=2k-1$.  That
is, $\chi'_s(G)\ge 2k$.
\end{proof}

\begin{lem}
\label{lem2}
Fix integers $k\ge 2$ and $g\ge 3$ and $n\ge g$.
If also $n\ge \lceil3*(k-1)^{g-1}/(k-2)\rceil$ when $k\ge 3$, then
there exists a simple $k$-regular bipartite graph on $2n$ vertices
with girth at~least~$g$.
\end{lem}

Erd\H{o}s and Sachs~\cite{ES,sachs} each proved the existence of regular graphs with
arbitrary degree and arbitrary girth.  We follow the outline of~\cite{ES}
(see~\cite[Theorem~III.$1.4'$]{bollobas_extremal-GT-book}), but
we must adapt the proof to ensure that $G$ is also bipartite.

\begin{proof}
Fix $k$, $g$, and $n$ as in the lemma.
Our proof is by induction on $k$.  The base case, $k=2$, holds by letting 
$G$ be a Hamiltonian cycle on $2n$ vertices.  For the induction step, let
$G$ be a $(k-1)$-regular bipartite graph on $2n$ vertices with girth at least
$g$.  
For each $A\subseteq
E(\overline{G})$, we write \Emph{$G+A$} to denote the graph formed from $G$
by adding each edge in $A$.  We iteratively build an edge set \Emph{$A$}
such that $G+A$ is $k$-regular, bipartite, and has girth at least $g$.  
Since $G$ is bipartite, denote its parts by $X$
and $Y$\aaside{$X$, $Y$}{0mm}.  Given $A$, let 
$\Xlow:=\{x\in X|~d_{G+A}(x)=k-1\}$\aaside{$\Xlow$, $\Xhigh$}{4mm} and 
$\Xhigh:=\{x\in X|~d_{G+A}(x)=k\}$.
Define $\Ylow$ and $\Yhigh$\aaside{$\Ylow$, $\Yhigh$}{4mm} analogously.  Note, for
each $A$, that $\Xlow,\Xhigh$ partition $X$ and $\Ylow,\Yhigh$ partition $Y$. 
Since $G+A$ is bipartite, also $|\Xlow|=|\Ylow|$ and $|\Xhigh|=|\Yhigh|$.  For
all $v,w\in V(G)$, denote by
\Emph{$\dist(v,w)$} the distance in $G+A$ from $v$~to~$w$.
For $W_1\subseteq V(G)$ and $W_2\subseteq V(G)$, let $\dist(W_1,W_2):=\min_{w_1\in
W_1, w_2\in W_2}\dist(w_1,w_2)$.\aside{$\dist(W_1,W_2)$}

Initially, let $A=\emptyset$.  If $|A|<n$, then we will show how to enlarge $A$,
either by adding a single edge, or by removing one edge and adding two.

\begin{figure}[!t]
\centering
\begin{tikzpicture}[thick, scale=.625]
\tikzstyle{uStyle}=[shape = circle, minimum size = 4.5pt, inner sep = 0pt,
outer sep = 0pt, draw, fill=none, semithick]
\tikzset{every node/.style=uStyle}

\draw (0,0) ellipse (1.5in and .25in);
\draw (-1.5,0.95) node[draw=none] {\footnotesize{$\Xlow$}};
\draw (1.5,0.95) node[draw=none] {\footnotesize{$\Xhigh$}};
\draw (0,.25in) -- (0,-.25in);
\draw (-.15in, -.05in) node {} ++ (0,.15in) node[draw=none]
{\footnotesize{$x_\ell$}};
\draw (.95in, -.05in) node (x3) {} ++ (0,.15in) node[draw=none, fill=none] {\footnotesize{$x_h$}};

\draw (-3.5,-1.25) node[draw=none] {\footnotesize{$G+A$}};
\def\low{-2.5cm}
\begin{scope}[yshift=\low]
\draw ellipse (1.5in and .25in);

\draw (-1.5,-.95) node[draw=none] {\footnotesize{$\Ylow$}};
\draw (1.5,-.95) node[draw=none] {\footnotesize{$\Yhigh$}};
\draw (0,.25in) -- (0,-.25in);
\draw (-.15in, .05in) node {} ++ (0,-.15in) node[draw=none]
{\footnotesize{$y_\ell$}};
\draw (.95in, .05in) node (y3) {} ++ (0,-.15in) node[draw=none, fill=none] {\footnotesize{$y_h$}};
\end{scope}
\draw (x3) -- (y3);

\begin{scope}[xshift=4in]

\draw (0,0) ellipse (1.5in and .25in);
\draw (-1.5,0.95) node[draw=none] {\footnotesize{$\Xlow$}};
\draw (1.5,0.95) node[draw=none] {\footnotesize{$\Xhigh$}};
\draw (-.75,.25in) -- (-.75,-.25in);
\draw (-.15in, -.05in) node (x2) {} ++ (0,.15in) node[draw=none]
{\footnotesize{$x_\ell$}};
\draw (.95in, -.05in) node (x3) {} ++ (0,.15in) node[draw=none, fill=none] {\footnotesize{$x_h$}};

\draw (3.5,-1.25) node[draw=none] {\footnotesize{$G+A'$}};
\def\low{-2.5cm}
\begin{scope}[yshift=\low]
\draw ellipse (1.5in and .25in);

\draw (-1.5,-.95) node[draw=none] {\footnotesize{$\Ylow$}};
\draw (1.5,-.95) node[draw=none] {\footnotesize{$\Yhigh$}};
\draw (-.75,.25in) -- (-.75,-.25in);
\draw (-.15in, .05in) node (y2) {} ++ (0,-.15in) node[draw=none]
{\footnotesize{$y_\ell$}};
\draw (.95in, .05in) node (y3) {} ++ (0,-.15in) node[draw=none, fill=none] {\footnotesize{$y_h$}};
\end{scope}
\draw[dashed] (x3) -- (y3);
\draw (x2) -- (y3) (y2) -- (x3);
\end{scope}

\end{tikzpicture}
\caption{When we cannot simply add an edge to $A$, we form $A'$ from $A$ by
removing edge $x_hy_h$ and adding both $x_{\ell}y_h$ and $y_{\ell}x_h$.  (For clarity,
most edges between $X$ and $Y$ are omitted.)\label{fig1}}
\end{figure}
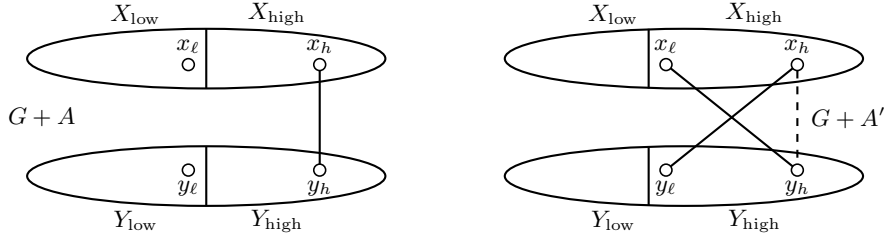

If $\Xlow=\emptyset$, then we are done.  So assume 
both $\Xlow$ and $\Ylow$ are nonempty.
If there exist $x_\ell\in
\Xlow$ and $y_\ell\in \Ylow$\aside{$x_\ell,y_\ell$}
such that $\dist(x_\ell,y_\ell)\ge
g-1$, then we add edge $x_{\ell}y_\ell$ (and are done).  So assume no such
$x_\ell, y_\ell$
exist.  The set of vertices at
distance no more than $g-2$ from any $x_\ell\in \Xlow$ has size at most
$1+(k-1)+(k-1)^2+\cdots+(k-1)^{g-2}<(k-1)^{g-1}/(k-2)$.  This set contains all
of $\Ylow$, 
so we assume $|\Xlow|=|\Ylow|<(k-1)^{g-1}/(k-2)$.
Note that $|\Xlow|<|X|$, 
so $|A|>0$. 
Fix arbitrary $x_\ell\in \Xlow$ and $y_\ell\in \Ylow$. 
We show there exists an edge
$x_hy_h\in A$\aaside{$x_h,y_h$}{-4mm} such that
$\dist(\{x_\ell,y_\ell\},\{x_h,y_h\})\ge g-1$; see Figure~\ref{fig1}.  
Let $\Abad$ denote the set\aaside{$\Abad$}{-4mm}
of edges in $A$ that fail
this criteria; note that $|\Abad|< 2(k-1)^{g-1}/(k-2)$.
Since $|X|\ge \lceil3*(k-1)^{g-1}/(k-2)\rceil$ and
$|\Xlow|<(k-1)^{g-1}/(k-2)$, we have
$|A|-|\Abad|=|\Xhigh|-|\Abad|=|X|-|\Xlow|-|\Abad| >
3*(k-1)^{g-1}/(k-2)-(k-1)^{g-1}/(k-2)-2(k-1)^{g-1}/(k-2)=0$. 
Thus, the desired edge $x_hy_h\in A$ exists.

Form $A'$ from $A$ by removing $x_hy_h$ and adding edges $x_{\ell}y_h$ and
$y_{\ell}x_h$.
Evidently, $|A'|=|A|+1$ and $G+A'$ is bipartite with maximum degree $k$.
Thus, it suffices to check that $G+A'$ has girth at least $g$.  By construction,
each of $x_\ell,y_\ell$ is distance at least $g-1$ from each of $x_h,y_h$ so any cycle
$C$ of length less than $g$ in $G+A'$ must use both of edges $x_{\ell}y_h$ and
$y_{\ell}x_h$.
Since $x_hy_h\in A$ and $G+A$ has girth at least $g$, every $x_h,y_h$-path in
$G+A-x_hy_h$ has length at least $g-1$.  Thus, $C$ contains vertices $x_{\ell},
x_h, y_\ell, y_h$ in that cyclic order.  But this contradicts that $C$ has
length less than $g$, since (by construction) every $x_{\ell},x_h$-path in
$G+A$ has length at least $g-1$.
\end{proof}

Conjecture~\ref{conj2} below slightly weakens
Conjecture~\ref{conj1}(6), and generalizes it to~graphs~with~$\Delta=k$.

\begin{conj}
\label{conj2}
For each integer $k\ge 3$,
there exists a girth $g_k$ such that if $G$ is bipartite with girth at least
$g_k$, with $\Delta(G)=k$, and with $m$ edges, then $\chi'_s(G)\le 2k$ and $G$
has a strong edge-coloring with colors $1,\ldots,2k$ that uses color $2k$ on at
most $m-(2k-1)\lfloor m/(2k-1)\rfloor$ edges.
\end{conj}

\section*{Acknowledgments}
Thanks to two anonymous referees for helpful comments which improved the
presentation.

\bibliographystyle{habbrv}
{{\bibliography{GraphColoring}}

\end{document}